\documentclass[11pt,a4paper]{article}

\usepackage{amsmath}
\usepackage{amsthm}
\usepackage{amsfonts}
\usepackage{paralist}
\usepackage{graphicx}
\usepackage{tikz}
\usetikzlibrary{arrows}

\usepackage{hyperref}
\hypersetup{%
plainpages=false,
colorlinks,urlcolor=blue,linkcolor=blue,citecolor=blue,
pdftitle={Optimal control of storage incorporating market impact and
  with energy applications},
pdfauthor={James Cruise, Lisa Flatley, Richard Gibbens and Stan Zachary},
pdfsubject={},
pdfpagemode={UseOutlines},
pdfpagelayout={OneColumn},
pdfstartview={FitBH}
}

\parskip\smallskipamount
\newtheorem{theorem}{Theorem}
\theoremstyle{remark}
\newtheorem{remark}{Remark}
\newtheorem{example}{Example}

\newcommand{\E}{\mathbf{E}}
\newcommand{\R}{\mathbb{R}}

\newcommand{\F}{\mathcal{F}}
\newcommand{\X}{X}
\newcommand{\oT}{\overline{T}}

\parskip \smallskipamount

\title{Optimal control of storage incorporating market impact\\ and
  with energy applications}

\author{James Cruise\footnote{Heriot-Watt University.  Research
    supported by EPSRC grant EP/I017054/1},
  Lisa Flatley\footnote{University of Warwick.  Research supported by
    EPSRC grant EP/K002228/1}, 
  Richard Gibbens\footnote{University of Cambridge.  Research
    supported by EPSRC grant EP/I016023/1}
  and Stan Zachary\footnotemark[1]}
\date{\today}

\begin{document}

\maketitle


\begin{abstract}
  Large scale electricity storage is set to play an increasingly
  important role in the management of future energy networks.  A major
  aspect of the economics of such projects is captured in arbitrage,
  i.e.\ buying electricity when it is cheap and selling it when it is
  expensive.  We consider a mathematical model which may account for
  nonlinear---and possibly stochastically evolving---cost functions,
  market impact, input and output rate constraints and both
  time-dependent and time-independent inefficiencies or losses in the
  storage process.  Our main concern is to develop the associated
  strong Lagrangian theory.  The Lagrange multipliers associated with
  the capacity constraints in particular have important economic
  interpretations with regard to the dimensioning of storage---both
  with respect to its capacity and its rate constraints---and prove
  key to the efficient control of a store.  We also develop an
  algorithm which determines, sequentially in time, both these
  Lagrange multipliers and the optimal control.  This algorithm
  further identifies, for each point in time, a time horizon beyond
  which it is not necessary to look in order to identify the optimal
  control at that point; this horizon is furthermore the shortest
  such.  The algorithm is thus particularly suitable for the
  management of storage over extended periods of time.  We give
  examples related to the management of real-world systems.  Finally
  we consider a pragmatic approach to the real-time management of
  storage in a stochastic cost environment, which is computationally
  feasible, optimal under certain ideal conditions, and which may in
  general be expected to perform close to optimally.  Our results are
  formulated in a general setting which permits their application to
  other energy management problems, and to other commodity storage
  problems.
\end{abstract}

\section{Introduction}
\label{sec:introduction}

How should one optimally control an energy store which is used to make
money by buying electricity when it is cheap, and selling it when it
is expensive?  While in its simplest form this is a classical
mathematical problem (see~\cite{Cahn} and, for early dynamic programming
approaches, \cite{Bell} and \cite{Drey}) we are interested in the
problem where the store has both finite capacity and rate constraints,
and where we allow that the activities of the store are of a
sufficient magnitude as to impact upon prices in the market in which
it operates.  The underlying mathematics thus required has various
novel features and needs to be carefully formulated so as to properly
account for physical characteristics of different storage technologies
and to deal with inherent nonlinearities which occur when prices are
impacted by the store's behaviour.

A closely related application is to the management of demand in such
systems, where the ability to contract with consumers to postpone
demand may be regarded as negative storage.  For some recent
discussion and work on these applications see, for example,
\cite{AVD,GJM,HCBJ, KHT,PADS,VHMS,WW} and the references therein; for
work on the optimal placement of storage within a network, see
\cite{TBH1,TBH2}.
These works are concerned, as here, with the mathematics of storage
for \emph{arbitrage}, i.e.\ taking advantage of---and hence assisting
in smoothing----price fluctuations over time.  This mathematics is of
course also quite generally applicable to the use of storage in other
markets.  (For the mathematics of other uses of storage in energy
systems---notably for buffering against uncertainty---see, for
example, \cite{BI1,BI2,BGK,GTL,HPSPB,HMN,WW}.)

We think of the available storage as a single store.  Its \emph{value}
is equal to the profit which can be made by a notional store ``owner''
buying and selling as above.  Our particular interest is in the case
where the activities of the store are sufficiently significant as to
have a market impact (the store becomes a ``price-maker'').  In this
case the store owner sees nonlinear cost functions as, at any time,
the marginal costs of buying or returns from selling vary with the
amount being bought or sold.  In the case where the \emph{system} or
\emph{societal} value of the store is required, this may be similarly
calculated by adjusting the notional buying and selling prices so that
the store ``owner'' is required to bear also the external costs of the
store's activities (see below for further discussion of this).

The nonlinearity of the cost functions means that the linear
programming techniques which might otherwise be used in the solution
of this problem are not generally available.  (However, see
Section~\ref{sec:problem} for some further discussion and references
for the case where linear programming techniques may be used.)
Neither are dynamic programming techniques---deterministic or
stochastic (see, for example, \cite{Bert76,Bert79})---always
tractable in practice.  The reason for the latter is that optimization
is typically over extended periods of time, during which the costs
involved usually vary with time in an irregular manner.  The
computational complexity of a dynamic programming approach may
therefore be unduly burdensome and is almost certainly so in a
stochastic environment.  Further, in the presence of temporal
heterogeneity dynamic programming approaches may fail to provide
necessary insights---for example, concerning the time horizons
necessary for optimal decision making, or sensitivities with respect
to local cost variations.

In the present paper we develop an approach based on the use of strong
Lagrangian techniques (convex optimization theory) which naturally
accommodates nonlinear cost functions, input and output rate
constraints, and temporal heterogeneity, and for which the associated
Lagrange multipliers provide the information necessary for the correct
dimensioning of storage with respect to both capacity and rate
constraints, and for the assessment of the economics of storage in
networks.  The strong Lagrangian approach also enables the development
of an algorithm for the solution of the problem which is efficient in
the sense that the decisions to be made at each point in time
typically depend only on a very short future horizon---which is
identifiable, but not determined in advance.  The length of this
horizon (the definition of which we make precise in
Section~\ref{sec:algorithm}) depends on the parameters of the store
and is of the same order as that of the
shortest period of time over which prices fluctuate significantly;
this is important when we may wish to optimally manage a store over a
very much longer, or perhaps indefinite, period of time.  Our approach
also allows us to account for differences in buying and selling prices
and for both time-dependent and time-independent inefficiencies in the
storage process.

Initially we work in a deterministic setting in which we assume that
all relevant buying and selling prices are known in advance.  For many
applications this is reasonable: as indicated above (and in the
realistic examples of Section~\ref{sec:example}) the time horizon
required for optimal decision making may be short.  However, elsewhere
there is a need to take account of stochastic variation, and
in Section~\ref{sec:stochastic-model} we consider prices which evolve
stochastically.  We show that in a somewhat idealised stochastic
setting---in which uncertainty evolves backwards in time as a
martingale---the optimal control is simply to replace future costs
by their expected values and to proceed as in the deterministic case.
We argue also that that this approach should continue to work well in
a more general stochastic setting when combined with the possibility
of re-optimisation at each time step.

In Section~\ref{sec:problem} we formally define the relevant
mathematical problem, while in Section~\ref{sec:solution} we use
strong Lagrangian theory to characterise mathematically its optimal
solution.  We use this theory in Section~\ref{sec:algorithm} to
develop the algorithm for the solution referred to above and to
characterise the evolving time horizon required for decision making in
a dynamic environment.  In Section~\ref{sec:sens-store-value} we show
how the value of the store changes with respect to variation in its
characteristic parameters.  Section~\ref{sec:example} considers
examples based on real data for UK electricity prices.
Section~\ref{sec:stochastic-model} studies models in which the cost
functions vary stochastically as described above, and proposes an
approach which we believe is as realistic as is practicable for
many applications.

\section{Problem formulation}
\label{sec:problem}

We work in discrete time, which we take to be integer.  We assume that
the store has total \emph{capacity} of $E$ (which, in the context of
an energy system, would be total energy which could be stored) and
input and output \emph{rate constraints} of $P_i$ and $P_o$
respectively (which, for an energy system, would be in units of
power).  We consider two types of (in)efficiency associated with the
store.  The first of these (and usually much the more significant in
practice) is a \emph{time-independent efficiency} $\eta$ which may be
defined as the fraction of energy bought which is available to sell.
This may be incorporated directly into the cost functions $C_t$, by
suitably rescaling selling and buying prices.
The second type of (in)efficiency may be regarded as \emph{leakage}
over time, and is modelled by assuming that at each successive time
instant there is lost a fraction $1-\rho$ of whatever is in the store
at that time.  We remark that it would also be possible to assume,
without loss of generality, that there was no leakage, i.e.\ that
$\rho=1$; this could be achieved by adjusting by a factor $\rho^t$ the
units of measurement of the volume in storage at each time~$t$ and
suitably redefining cost functions and constraints; however, there is
very little effort saved by introducing this additional level of
abstraction, and so we in general avoid doing so. 

Let $\X=\{x:-P_o\le x\le P_i\}$.  Both buying and selling prices at
time $t$ may conveniently be represented by a cost function $C_t$,
which we assume to be convex, and is such that $C_t(x)$
is the cost at time~$t$ of increasing the level of the store contents
(after any leakage---see below)
by $x$, positive or negative.  Typically---in a conventional
store and with positive prices---we have that each function $C_t$ is
increasing and that $C_t(0)=0$; then, for positive $x$, $C_t(x)$ is
the cost of buying $x$ units (for example of energy) and, for negative
$x$, $C_t(x)$ is the negative of the reward for selling $-x$ units;
however, for some applications (see below), the interpretation of the
functions~$C_t$ may vary slightly from this, and only the convexity
condition on these functions is required.
This convexity assumption corresponds, for each time $t$, to an
increasing cost to the store of buying each additional unit, a
decreasing revenue obtained for selling each additional unit, and
every unit buying price being at least as great as every unit selling
price.  Note that incorporating the time-independent (or
``round-trip'') efficiency~$\eta$ into the cost functions~$C_t$, as
discussed above, 
automatically preserves convexity whenever these cost functions are
increasing.  For a discussion of non-convex cost functions, see for
example~\cite{FMW}.

We are not concerned here to discuss the market derivation of the
functions $C_t$, for a discussion of which see, for
example,~\cite{CFZ}.

As indicated above, if the problem is to determine the value of the
store to the entire system in which it operates, or to society, then
these prices are taken to be those appropriate to the system or to be
societal costs.  Thus, for example, for $x$ positive, $C_t(x)$ may be the
price paid by the store at time $t$ for $x$ units of, for example,
energy plus the increased cost paid by other energy users at that time
as a result of the store's purchase increasing market prices---again
see~\cite{CFZ} for a detailed explanation of how the current model may
be used in this context.

Figure~\ref{fig:cost_fn} thus illustrates a typical cost function
$C_t$.  While the function~$C_t$ may be formally regarded as defined
over the whole real line, the rate constraints means that for the
purposes of the present problem its domain is effectively restricted
to the set $\X$ defined above.  (We shall later wish to consider the
effect of varying the rate constraints.)

\begin{figure}[!ht]
  \centering
\begin{tikzpicture}[scale=1.0]
  \draw[->,>=stealth'] (-5,0) -- (5,0) node [right] {$x$};
  \draw[->,>=stealth'] (0,-1.5) -- (0,5) node [left] {$C_t(x)$};
  \draw[very thick] (-4,-1) .. controls (-3,-0.9) .. (0,0)
  .. controls (3,1.5) .. (4,3.25);
  \draw[dashed,very thick] (-5,-1.06) -- (-4,-1);
  \draw[dashed,very thick] (4,3.25) -- (5,5);
  \draw[dashed, thick] (-4,0) node [above] {$-P_0$} -- (-4,-1);
  \draw[dashed, thick] (4,0) node [below] {$P_i$} -- (4,5);
  \draw (0.1,0) node [below left] {$0$};
  \draw (-2,-1) node [below] {sell};
  \draw (2,-1) node [below] {buy};
 \end{tikzpicture}
  \caption{Illustrative cost function $C_t$.  The domain of the
    function is effectively restricted to the set
    $\X=\{x:-P_o\le x\le P_i\}$.}
  \label{fig:cost_fn}
\end{figure}
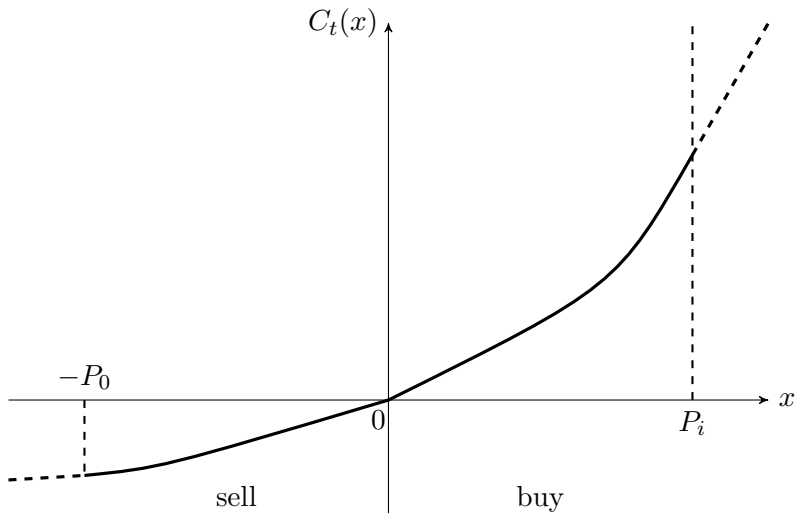
A special case is that of a ``small'' store, whose operations do not
influence the market (the store is a ``price-taker'' rather than a
``price-maker''), and which at time $t$ buys and sells at given prices
per unit of $c^{(b)}_t$ and $c^{(s)}_t$ respectively, where we assume
that $c^{(b)}_t\ge c^{(s)}_t$.  Here the function $C_t$ is given by
\begin{equation}
  \label{eq:1}
  C_t(x) =
  \begin{cases}
    c^{(b)}_t x & \quad\text{if $x\ge0$}\\
    c^{(s)}_t x & \quad\text{if $x<0$}.
  \end{cases}
\end{equation}

Finally, we assume for the moment that all prices are known in
advance, so that the problem of controlling the store is
deterministic.  We consider a realistic stochastic model in
Section~\ref{sec:stochastic-model}.


Denote the successive levels of the store by a vector
$S=(S_0,\dots,S_T)$ where $S_t$ is the level of the store at each
successive time $t$.  Define also the vector
$x(S)=(x_1(S),\dots,x_T(S))$ by $x_t(S)=S_t-\rho S_{t-1}$ for each
$t\ge1$.  Here $\rho$ is the leakage measure defined above, so that
$x_t(S)$ represents the addition to the store at time~$t$.  It is
convenient to assume that both the initial level $S_0$ and the final
level $S_T$ of the store are fixed in advance at $S_0=S^*_0$ and
$S_T=S^*_T$.  (If the final level $S_T$ is not fixed and the cost
function $C_T$ is strictly increasing, then, for an optimal control,
we may take $S_T$ to be minimised---so that finally as much as
possible of the contents of the store are sold; however, we might, for
example, wish to require $S^*_T=S^*_0$ in order to solve a problem in
which the cost functions varied cyclically.)


The problem thus becomes:

\begin{compactitem}
\item[$\mathbf{P}$:]
  (given the convex functions $C_t$) choose $S$ so as to minimise
  \begin{equation}
    \label{eq:2}
    G(S) := \sum_{t=1}^T C_t(x_t(S))
  \end{equation}
  subject to the capacity
  constraints
  \begin{gather}
    \label{eq:3}
    S_0=S^*_0, \qquad S_T=S^*_T, \qquad
    0 \le S_t\le E,
    \quad 1 \le t \le T-1.
  \end{gather}
  and the rate constraints
  \begin{equation}
    \label{eq:4}
    x_t(S) \in \X,
    \qquad 1 \le t \le T.
  \end{equation}
\end{compactitem}
We shall say that a vector $S$ is \emph{feasible} for the
problem~$\mathbf{P}$ if it satisfies both the capacity
constraints~\eqref{eq:3} and the rate constraints~\eqref{eq:4}.  We
shall assume that $S^*_0$ and $S^*_T$ are sufficiently close that it
is possible to change the level of the store from $S^*_0$ to $S^*_T$
between times $0$ and $T$, i.e.\ that the set of feasible vectors~$S$
is nonempty.  Note that this set is then closed and convex and that
the function~$G$ defined by~\eqref{eq:2} is convex, and strictly so
when the functions~$C_t$ are strictly convex.  Hence a solution to the
problem~$\mathbf{P}$ always exists, and is unique when the
functions~$C_t$ are strictly convex.

In the case where the cost functions $C_t$ are linear, or piecewise
linear, as in the ``small store'' case given by \eqref{eq:1},
the problem~$\mathbf{P}$ may be reformulated as a linear programming
problem, and solved by, for example, the use of the minimum cost
circulation algorithm (see, for example, \cite{Boyd,AMO}).  Our aim in
the present paper is to deal with the general case, to develop the
related Lagrangian theory together with an algorithm which identifies
both problem solution and associated Lagrange multipliers, and to use
this algorithm to show that the optimal choice of $S_t$ at each
time~$t$ depends only on a typically very short time horizon, thus
providing an efficient approach to the solution of the problem
(particularly the real-time management of the store within
applications) over long time periods.

Finally, we note that the mathematical problem formulated in this
section is applicable to physical problems---in energy management and
elsewhere---other than those of conventional storage.  One such is the
management of ``one-sided'' storage, such as hydroelectric power, in
which inputs are predetermined and (we assume here) known and the only
control is over the output at each successive time~$t$.  Here the
control remains the sequence $S$ of successive levels of the store,
and, for each $t$, the function~$C_t$ is such that $C_t(x_t(S))$
remains the cost of the “change” $x_t(S)$ as defined earlier.  It may
not here be natural to have $C_t(0)=0$, and we may wish to allow the
space $X$ of feasible values of $x_t(S)$ to depend on the
time~$t$---something which causes no additional complications.

A further possible application might be to the buffering of
\emph{demand}, which, as remarked earlier, may be regarded as negative
storage, $S_t$ now being the amount of demand ``postponed'' at each
successive time~$t$.  The cost functions~$C_t$ would represent the
costs of such postponement.  However, to be realistic such costs would
probably also need to reflect the durations of such postponements.

\section{Lagrangian formulation and characterisation of solution}
\label{sec:solution}

We develop the strong Lagrangian theory~\cite{Boyd,Whi} associated
with the problem~$\mathbf{P}$ defined above.  Theorem~\ref{thm:ls}
gives sufficient conditions for a value $S^*$ of $S$ to solve the
problem, while Theorem~\ref{thm:exist} guarantees the existence of
such a value of $S^*$, together with the associated vector (cumulative
Lagrange multiplier)~$\mu^*$ defined there.

\begin{theorem}
  \label{thm:ls}
  Suppose that there exists a
  vector $\mu^*=(\mu^*_1,\dots,\mu^*_T)$ and a value
  $S^*=(S^*_0,\dots,S^*_T)$ of $S$ such that
  \begin{compactenum}[(i)]
  \item $S^*$ is feasible for the stated problem,
  \item for each $t$ with $1\le t\le T$, $x_t(S^*)$ minimises
    $C_t(x)-\mu^*_tx$ in $x\in\X$,
  \item the pair $(S^*,\mu^*)$ satisfies the complementary slackness
    conditions, for $1\le t\le T-1$,
    \begin{equation}
      \label{eq:5}
      \begin{cases}
        \rho\mu^*_{t+1} = \mu^*_t & \quad\text{if $0 < S^*_t < E$,}\\
        \rho\mu^*_{t+1} \le \mu^*_t & \quad\text{if $S^*_t = 0$,}\\
        \rho\mu^*_{t+1} \ge \mu^*_t & \quad\text{if $S^*_t = E$.}
      \end{cases}
    \end{equation}
  \end{compactenum}
  Then $S^*$ solves the stated problem~$\mathbf{P}$.
\end{theorem}

\begin{proof}
  Let $S$ be any vector which is feasible for the problem (with
  $S_0=S^*_0$ and $S_T=S^*_T$).  Then, from
  the condition~(ii),
  \begin{displaymath}
    \sum_{t=1}^T\left[C_t(x_t(S^*)) - \mu^*_t x_t(S^*)\right]
    \le
    \sum_{t=1}^T\left[C_t(x_t(S)) - \mu^*_t x_t(S)\right].
  \end{displaymath}
  Rearranging and recalling that $S$ and $S^*$ agree at $0$ and at
  $T$, we have
  \begin{align*}
    \sum_{t=1}^T C_t(x_t(S^*)) - \sum_{t=1}^T C_t(x_t(S))
    & \le \sum_{t=1}^T \mu^*_t (S^*_t - \rho S^*_{t-1} - S_t +
    \rho S_{t-1})\\
    & = \sum_{t=1}^{T-1}(S^*_t - S_t)(\mu^*_t - \rho\mu^*_{t+1})\\
    & \le 0,
  \end{align*}
  by the condition~(iii), so that the result follows.  
\end{proof}

\begin{remark}
  Note that when the functions $C_t$ are increasing the vector~$\mu^*$
  of Theorem~\ref{thm:ls} may be taken to be \emph{nonnegative}, i.e.\
  to have nonnegative components: if $\mu^*$ does not satisfy this
  condition then its negative components may all be increased to $0$
  and the pair $(S^*,\,\mu^*)$ will continue to satisfy the conditions
  of the theorem.
\end{remark}

The vector $\mu^*$ is a cumulative form of the vector of
Lagrange multipliers associated with the capacity constraints
\eqref{eq:3} (see the proof of Theorem~\ref{thm:exist} below).  It has
the interpretation that, for each $t$, the quantity $\mu^*_t$ may be
regarded as a notional reference value per unit volume in storage
\emph{at that time}.  Thus, in the condition~(ii) of the theorem,
$C_t(x)$ is the cost at time~$t$ of increasing the level of the store
by $x$ (again positive or negative) and $\mu^*_tx$ may be regarded as
a current offsetting measure of value added to the store; the quantity
$C_t(x)-\mu^*_tx$ is thus to be minimised in $x\in\X$.  The
relations~\eqref{eq:5} of condition~(iii) of the theorem are then such
that, were they to be violated, $x_t$ and $x_{t+1}$ could in general
be adjusted so as to leave unchanged the level of the store at the end
of time~$t+1$ while reducing the overall cost of operating the store
throughout the period consisting of the times $t$ and $t+1$.

Note also that, in the condition~(ii) of Theorem~\ref{thm:ls}, the
minimisation takes place without reference to the \emph{capacity}
constraints (as is appropriate given the above Lagrangian
interpretation of $\mu^*$).  However, the minimisation of that
condition is required to respect the \emph{rate} constraints
$x\in\X$---for which no Lagrange multiplier is introduced at this
stage (but see Section~\ref{sec:sens-store-value}).  The reason for
the apparent asymmetry of treatment of the two constraint types is
that it is only the capacity constraints which introduce complexity
into the optimisation problem, by introducing interactions between the
amounts which may be bought and sold at different times.  The rate
constraints could, if we wished, be dropped from the formal statement
of the problem by suitably modifying the cost functions so that the
violation of these constraints was simply prohibitively expensive.

Before considering Theorem~\ref{thm:exist}, which guarantees the
existence of the pair $(S^*,\mu^*)$, we give a couple of simple
examples, in each of which the reference vector $\mu^*$ is identified.
Theorem~\ref{thm:ls} is not, however, needed for the solution of the
first, very simple, example.  It is needed in the second example only
in the case where the store is sufficiently large as to have market
impact (i.e.\ be a price-maker).

\begin{example}
  As a simple (toy) example, suppose that $T=2$ and that the cost
  functions $C_t$, $t=1,2$, in addition to being increasing and
  convex, are differentiable (with necessarily continuous first
  derivatives); however, as an exception and in order to allow for a
  distinction between buying and selling prices we allow a difference
  between the left and right derivatives of the functions~$C_t$ at
  $0$, denoting these one-sided derivatives by $C'_t(0-)$ and
  $C'_t(0+)$ respectively (with, necessarily, $C'_t(0-)\le C'_t(0+)$
  for $t=1,2$).  We suppose additionally, and again for simplicity,
  that the input and output rate constraints are equal, setting
  $P_i=P_o=P$, and that there is no leakage (i.e.\ $\rho=1$).
  Finally we suppose $S^*_0=S^*_2=0$ so that the store starts empty
  and is required to finish empty.  Thus the only possible control of
  the store lies in the choice of the amount~$x\ge0$ which is bought
  at time~$1$ and sold again at time~$2$.

  For this example, the optimal policy is of course easily
  determined. Our concern is merely to identify, in this very simple
  case, the vector $\mu^*$ of Theorem~\ref{thm:ls}.  This vector plays
  a crucial r\^ole in more complex optimization over longer time
  periods.  We consider the three possible cases.
  \begin{compactenum}[(i)]
  \item If $C'_1(0+)\ge C'_2(0-)$ then clearly the optimal policy is
    buy and sell nothing and we take $x=0$.  For the vector~$\mu^*$ of
    Theorem~\ref{thm:ls} we may take $\mu^*_1=C_1'(0+)$ and
    $\mu^*_2=C_2'(0-)$.
  \item If $C'_1(0+)<C'_2(0-)$ and there exists $x$ such that both
    $0\le x\le\min(E,\,P)$ and $C'_1(x)=C'_2(-x)$, then this choice of
    $x$ is again clearly optimal.  The vector~$\mu^*$ is given
    (uniquely) by $\mu^*_1=\mu^*_2=C'_1(x)$.
  \item Finally, if $C'_1(x)<C'_2(-x)$ for all $x$ such that $0\le
    x\le\min(E,\,P)$, then the optimal choice of $x$ is given by
    $x=\min(E,\,P)$.  In the case where $P\le E$ we require
    $C'_1(P)\le\mu^*_1=\mu^*_2\le C'_2(-P)$, while in the case where
    $E<P$ we require $C'_1(E)=\mu^*_1\le\mu^*_2=C'_2(-E)$.
  \end{compactenum}
  Note that the actual solution to this very simple problem depends on
  $E$ and $P$ only through $\min(E,\,P)$.  However, as previously
  observed, $\mu^*$ plays an asymmetric r\^ole with respect to
  capacity and rate constraints and thus formally differs in the
  case~(iii) according to which of $E$ or $P$ is the greater.
\end{example}

\begin{example}
  \emph{Periodic costs.}  As a second simple example, we suppose that
  the cost functions vary over time in a manner which is completely
  periodic.  To begin with, we consider the ``small store'', or
  price-taker, case in which the cost functions~$C_t$ are given by
  \eqref{eq:1} (with $c^{(b)}_t\ge c^{(s)}_t$ for all $t$).  We
  suppose that the periodic behaviour is such that, at some time $t_1$
  in a cycle, both $c^{(b)}_{t_1}$ and $c^{(s)}_{t_1}$ are
  simultaneously at a minimum; the unit costs $c^{(b)}_{t}$ and
  $c^{(s)}_{t}$ then increase monotonically up to a time $t_2>t_1$
  where they are simultaneously at a maximum, before decreasing
  monotonically again to the same minimum value as previously at
  further time $t_3>t_2$; this pattern is then repeated indefinitely
  with period $t_3-t_1$.  We suppose also that the minimum value of the
  unit buy costs $c^{(b)}_{t}$ is less than the maximum value of the
  unit sell costs $c^{(s)}_{t}$ (otherwise the store remains unused).
  We again assume, for simplicity, that there is no leakage (i.e.\
  $\rho=1$), that $P_i=P_o=P$ and that time is sufficiently
  finely discretised that $E/P$ (the minimum time in which the store
  may completely empty or fill) may be taken to be integer.  The
  optimal control policy depends (up to a multiplicative constant) on
  $E$ and $P$ only through the ratio $E/P$; hence, without loss of
  generality, we assume $P=1$.  

  The simplicity of this example is such that the optimal control of
  the store is again immediately clear: for all $E$ there exist
  reference costs $\mu^{(b)}\le\mu^{(s)}$ such that the store buys the
  maximum value of one unit at those times such that
  $c^{(b)}_t<\mu^{(b)}$ and sells the maximum value of one unit at
  those times such that $c^{(s)}_t>\mu^{(s)}$; for $E$ sufficiently
  small we may take $\mu^{(b)}<\mu^{(s)}$ and the store completely
  empties and fills on each cycle; however, as $E$ increases it
  reaches a value at which the reference costs $\mu^{(b)}$ and
  $\mu^{(s)}$ equalise, and for this and larger values of $E$ the
  capacity constraint is no longer binding.

  As in the case of the previous example, this ``small store'' problem
  is too simple for its solution to require the use of the reference
  vector $\mu^*$ of Theorem~\ref{thm:ls} (but see below for where it
  is needed).  We note, however, that this vector may be given by
  $\mu^*_t=\mu^{(b)}$ at those times~$t$ at which the store is buying,
  and by $\mu^*_t=\mu^{(s)}$ at those times~$t$ at which it is
  selling; at other times (at each of which the store will either be
  completely full or completely empty) $\mu^*_t$ is merely required to
  satisfy the condition~(iii) of Theorem~\ref{thm:ls} together with
  the condition $c^{(s)}_t\le\mu^*_t\le c^{(b)}_t$ (so that the
  condition~(ii) of Theorem~\ref{thm:ls} is satisfied).

  We also comment briefly on the effect of varying the frequency of
  the cost variation.  If, in what should strictly be a
  continuous-time setting, this frequency is increased by a factor
  $\alpha$ with the rate constraint~$P$ being similarly increased by
  the same factor, then this corresponds to a simple time speed-up,
  with the store's revenue per unit time also being increased by the
  factor~$\alpha$.  However, suppose instead that while the frequency
  of the cost variation is increased by the factor~$\alpha$, the rate
  constraint $P$ is held constant at its original value and that the
  capacity constraint $E$ is replaced by $E/\alpha$.  It then
  follows, from the earlier observation that the optimal control
  depends on $E$ and $P$ only through their ratio, that the optimal
  control is here a rescaled version of the original and that the
  store's revenue per unit time remains unchanged from the original.
  Thus we have the well-known result that more frequent cost variation
  enables the same revenue to be obtained with a smaller store
  capacity.

  When we consider the general case in which the store is a
  price-maker, and in which the cost functions~$C_t$ have the same
  general periodicity over time, but no longer have the simple
  structure given by~\eqref{eq:1}, then the store may fill and empty
  over periods of time which are longer than the minimum necessary, so
  as to avoid the higher costs or penalties of buying or selling too
  much at once.  The reference vector $\mu^*$ of Theorem~\ref{thm:ls}
  then becomes essential in deciding the correct volume of each
  transaction.
\end{example}

\medskip

Theorem~\ref{thm:ls} does not require the convexity of the cost
functions $C_t$ of the problem~$\mathbf{P}$ defined in
Section~\ref{sec:problem}.  This condition is, however, required to
ensure the existence of the vector $\mu^*$ of that theorem, as is
given by Theorem~\ref{thm:exist} below.  The latter theorem identifies
$\mu^*$ as essentially a cumulative Lagrange multiplier for capacity
constraint variation.  It is a further application of arguments to be
found in strong Lagrangian theory (again see~\cite{Whi}).

We have already observed that, under \emph{strict} convexity of the
cost functions~$C_t$, the solution~$S^*$ to the problem~$\mathbf{P}$
is unique.  However, we further remark that even this condition is
insufficient to guarantee uniqueness of $\mu^*$ as above.  We address
this issue in Section~\ref{sec:sens-store-value}, where we assume
sufficient differentiability conditions on the cost functions~$C_t$ as
to ensure uniqueness of $\mu^*$ and to derive sensitivity results for
variation of the minimised cost function of $\mathbf{P}$ with respect
to both its capacity and rate constraints.

Prior to Theorem~\ref{thm:exist} it is convenient to introduce the
more general problem~$\mathbf{P}(a,\,b)$ in which $S_0$ is kept fixed
at the value $S^*_0$ of interest above, but in which $S_1,\dots,S_T$
are allowed to vary between quite general upper and lower bounds:
\begin{compactitem}
\item[$\mathbf{P}(a,\,b)$:] 
  minimise $\sum_{t=1}^TC_t(x_t(S))$ over all $S=(S_0,\dots,S_T)$
  with $S_0=S^*_0$ and subject to the further constraints
  \begin{equation}
    \label{eq:6}
    a_t \le S_t \le b_t,  \qquad 1 \le t \le T,
  \end{equation}
  and $x_t(S)\in\X$ for $1 \le t \le T$, where $a=(a_1,\dots,a_T)$ and
  $b=(b_1,\dots,b_T)$ are such that $a_t\le b_t$ for all $t$.
\end{compactitem}
Note that the convexity of the functions $C_t$ guarantees their
continuity, and, since for each $a$, $b$ as above the space of allowed
values of $S$ is compact, a solution $S^*(a,\,b)$ to the
problem~$\mathbf{P}(a,\,b)$ always exists.  Let $V(a,\,b)$ be the
corresponding minimised value of the objective function, i.e.\
$V(a,\,b)=\sum_{t=1}^TC_t(x_t(S^*(a,\,b)))$.  Then $V(a,\,b)$ is
itself convex in $a$ and $b$.  (To see this, consider, for example,
any convex combination
$(\bar a,\bar b)=(\lambda a_1+(1-\lambda)a_2,\lambda b_1+(1-\lambda)b_2)$
of any two values $(a_1,b_1)$ and $(a_2,b_2)$ of the pair $(a,b)$, where
$0\le\lambda\le1$; the linearity of the constraints~\eqref{eq:3} and
\eqref{eq:4} implies that the vector
$\bar S=\lambda S^*(a_1,b_1)+(1-\lambda)S^*(a_2,b_2)$ is feasible for
the problem~$\mathbf{P}(\bar a,\,\bar b)$; hence
\begin{align*}
  V(\bar a,\bar b)
  & \le \sum_{t=1}^TC_t(x_t(\bar S))\\
  & = \sum_{t=1}^TC_t(\lambda x_t(S^*(a_1,b_1)) + (1-\lambda)x_t(S^*(a_2,b_2)))\\
  & \le \lambda\sum_{t=1}^TC_t(x_t(S^*(a_1,b_1)))
  + (1-\lambda)\sum_{t=1}^TC_t(x_t(S^*(a_2,b_2)))\\
  & = \lambda V(a_1,b_1) + (1-\lambda)V(a_2,b_2),  
\end{align*}
where the second inequality above follows from the convexity of the
functions $C_t$.)
Define also $a^*$ and $b^*$ to be the values of $a$ and $b$
corresponding to our particular problem~$\mathbf{P}$ of interest,
i.e.\ $a^*_t=0$ and $b^*_t=E$ for $1\le t\le T-1$, and
$a^*_T=b^*_T=S^*_T$.  Further, let
$S^*=(S^*_0,\dots,S^*_T)=S^*(a^*,\,b^*)$ denote the solution to this
problem.

\begin{theorem}
  \label{thm:exist}
  Under the given convexity condition on the cost functions $C_t$,
  there always exists a pair $(S^*,\mu^*)$ which solves the
  problem~$\mathbf{P}$ as in Theorem~\ref{thm:ls}.
\end{theorem}
\begin{proof}
  Consider the more general problem~$\mathbf{P}(a,\,b)$ defined above.
  Introduce slack (or surplus) variables $z=(z_1,\dots,z_t)$ and
  $w=(w_1,\dots,w_t)$ and rewrite this problem as:
  \begin{compactitem}
  \item[$\mathbf{P}(a,\,b)$:] 
    minimise $\sum_{t=1}^TC_t(x_t(S))$ over all $S=(S_0,\dots,S_T)$
    with $S_0=S^*_0$, all $z\ge0$, all $w\ge0$, and subject to the
    further constraints
    \begin{align}
      S_t - z_t & = a_t, \qquad 1 \le t \le T, \label{eq:7}\\
      S_t + w_t & = b_t, \qquad 1 \le t \le T, \label{eq:8}
    \end{align}
    and, again, $x_t(S)\in\X$ for $1 \le t \le T$. 
  \end{compactitem}

  Since, as already observed, the function $V(a,\,b)$ is itself convex
  in $a$ and $b$, it follows by the supporting hyperplane theorem (see
  \cite{Boyd} or \cite{Whi}), that there exist vectors (Lagrange
  multipliers) $\alpha^*=(\alpha^*_1,\dots,\alpha^*_T)$ and
  $\beta^*=(\beta^*_1,\dots,\beta^*_T)$ such that
  \begin{equation}
    \label{eq:9}
    V(a,\,b) \ge V(a^*,\,b^*) + \sum_{t=1}^T\alpha^*_t(a_t-a^*_t)  +
    \sum_{t=1}^T\beta^*_t(b_t-b^*_t)
    \qquad\text{for all $a$, $b$}.
  \end{equation}
  Thus also, for all $S$ with $S_0=S^*_0$ and such that $x_t(S)\in\X$
  for $1 \le t \le T$, for all $z\ge0$, and for all $w\ge0$,
  \begin{multline}
    \label{eq:10}
    \sum_{t=1}^T \left[C_t(x_t(S)) - \alpha^*_t(S_t-z_t)
      -\beta^*_t(S_t+w_t)\right]\\
    \ge
    \sum_{t=1}^T \left[C_t(x_t(S^*)) - \alpha^*_t(S^*_t-z^*_t)
      -\beta^*_t(S^*_t+w^*_t)\right]
  \end{multline}
  Since the components of $z$ and $w$ may take arbitrary positive
  values, we deduce immediately the following usual complementary
  slackness conditions for the vectors of Lagrange multipliers
  $\alpha^*$ and $\beta^*$:
  \begin{align}
    \alpha^*_t \ge 0, \qquad &
    \text{$\alpha^*_t=0$  whenever $z^*_t>0$},
    \qquad 1\le t \le T,
    \label{eq:11}\\
    \beta^*_t \le 0, \qquad &
    \text{$\beta^*_t=0$  whenever $w^*_t>0$},
    \qquad 1\le t \le T.
    \label{eq:12}
  \end{align}
  Thus, from \eqref{eq:10}--\eqref{eq:12} and by taking $z_t=w_t=0$ for
  all $t$ on the left side of~\eqref{eq:10}, it follows that, for all
  $S$ with $S_0=S^*_0$ and $x_t(S)\in\X$ for $1 \le t \le T$,
  \begin{equation}
    \label{eq:13}
    \sum_{t=1}^T \left[C_t(x_t(S)) - (\alpha^*_t+\beta^*_t)S_t \right]
    \ge
    \sum_{t=1}^T \left[C_t(x_t(S^*)) - (\alpha^*_t+\beta^*_t)S^*_t \right].
  \end{equation}
  Thus also, for all $x=(x_1,\dots,x_t)$ such that $x_t\in\X$ for $1
  \le t \le T$, by defining $S$ by $S_0=S^*_0$ and
  $S_t=\rho S_{t-1}+x_t$ for $1\le t\le T$, it follows that
  \begin{equation}
    \label{eq:14}
    \sum_{t=1}^T \left[C_t(x_t) - \mu^*_t x_t \right]
    \ge
    \sum_{t=1}^T \left[C_t(x_t(S^*)) -  \mu^*_t x_t(S^*) \right].
  \end{equation}
  where, for each $1\le t\le T$, we define
  \begin{equation}
    \label{eq:15}
    \mu^*_t=\sum_{u=t}^T\rho^{u-t}(\alpha^*_u+\beta^*_u).
  \end{equation}
  It now follows that the pair $(S^*,\mu^*)$ satisfies the
  conditions~(i) and (ii) of Theorem~\ref{thm:ls}.  Further, on
  recalling from \eqref{eq:7} and \eqref{eq:8} respectively that, for
  $1\le t\le T-1$, we have $z^*_t=0$ if and only if $S^*_t=0$ and
  $w^*_t=0$ if and only if $S^*_t=E$, it follows also from
  \eqref{eq:11}, \eqref{eq:12} and the definition~\eqref{eq:15} of the
  vector $\mu^*$, that the pair $(S^*,\mu^*)$ satisfies the
  complementary slackness conditions (iii) of Theorem~\ref{thm:ls}.
\end{proof}

Recall the earlier interpretation of each successive $\mu^*_t$ as
providing a unit reference value determining the quantity $x_t$
(positive or negative) which should be added to the level of the store
at that time.  In Section~\ref{sec:algorithm} we give an efficient
algorithm for the determination of the successive values of $\mu^*_t$.

\section{Determination of optimal control and associated Lagrange
  multipliers}
\label{sec:algorithm}

We now give an explicit construction of a pair $(S^*,\mu^*)$
as in Theorem~\ref{thm:ls}.  This construction further provides an
algorithm for the solution of the problem~$\mathbf{P}$ in the general
case.  The algorithm proceeds sequentially in time, and has the
``locality'' property that, at each time~$t$, the identification of
the optimal value $x^*_t$ of $x_t$ requires a knowledge of the cost
functions $C_{t'}$ only up to a time horizon which, while necessarily
greater than $t$, is frequently very much less than $T$.  Thus, for
example, if the cost functions vary strongly on an essentially daily
cycle, while the period over which the optimal control is required is
of the order of months or years, nevertheless the optimal decision at
each point in time typically depends only on a knowledge of the cost
functions for a future period of the order of a day or so---see the
further discussion at the end of this section and the examples of
Section~\ref{sec:example}.  The algorithm is thus
in general suitable for the optimal control of the store on an
essentially infinite time horizon.  We make these ideas clear below.

We assume for the moment that there is no leakage from the store over
time, i.e.\ that $\rho=1$.  With this assumption, the algorithm below
may briefly be described as that of attempting to choose $(S^*,\mu^*)$
so as to satisfy the conditions of Theorem~\ref{thm:ls}, by choosing
the components of these vectors successively in time and by keeping
$\mu^*_t$ as constant as possible over $t$, changes only being allowed
at those times when the store is either empty or full.  Once the
algorithm is understood, the modifications required to deal with the
more general case $\rho\le1$ are easily seen and are indicated in
brief at the end of this section.

For further simplicity, we suppose first that the cost functions $C_t$
are all strictly convex.  Then, as already noted, the vector~$S^*$ of
Theorem~\ref{thm:ls} is unique---though the corresponding
vector~$\mu^*$ need not be.  We give a construction of $(S^*,\mu^*)$
which is sequential in time.  For any $t$ such that $1\le t\le T$ and
any (scalar) $\mu$, define $x^*_t(\mu)$ to be the unique value of
$x$ which minimises $C_t(x)-\mu x$ in $x\in\X$.  Note that
$x^*_t(\mu)$ is then continuous and increasing (though not necessarily
strictly so) in $\mu$.  We show how to identify inductively a sequence
of times $0=T_0<T_1<\dots<T_K=T$ and a corresponding sequence
$(\bar\mu_1,\dots,\bar\mu_K)$, such that, for each $k=1,\dots,K$, we
may take $\mu^*_t=\bar\mu_k$ for $T_{k-1}+1\le t\le T_k$.  The vector
$S^*$ is then constructed as in (ii) of Theorem~\ref{thm:ls} and the
pair $(S^*,\mu^*)$ satisfies all the conditions of that theorem.

Further, for each $k=1,\dots,K-1$, we identify a time $\oT_k>T_k$ such
that, for any $t$,
\begin{compactenum}[1.]
\item whether or not $\oT_k$ is equal to $t$ is does not depend on the
  cost functions subsequent to time $t$;
\item whenever $\oT_k$ is equal to $t$, both the values of $T_k$ and
  of $(S^*_t,\mu^*_t)$ for $1\le t\le T_k$ do not depend on the cost
  functions subsequent to the time $t$; thus for each $t$ such that
  $T_{k-1}+1\le t\le T_k$, the time~$\oT_k$ represents the time horizon
  identified earlier as that beyond which it is not necessary to look
  for the determination of the optimal decision at time~$t$.
\end{compactenum}
Thus, were the cost functions stochastic, we should describe each
$\oT_k$ as a stopping time (though of course the nature of the optimal
control in a stochastic environment might well be different---see
Section~\ref{sec:stochastic-model}).  


In stating the construction it will be sufficient to consider the
identification of the time~$T_1$ and the constant~$\bar\mu_1$,
together with the further time~$\oT_1$.  Since the optimal control is
then identified up to the time $T_1$, the construction may then be
restarted at that time.  Theorem~\ref{thm:const} below then shows that
the pair $(S^*,\mu^*)$ thus constructed over the entire time period
$[1,\dots,T]$ has all the required properties necessary to define the
optimal control.

We thus consider trial values $\mu$ of $\bar\mu_1$.  
For each (scalar) $\mu$, define a vector
$S(\mu)=(S_0(\mu),\dots,S_T(\mu))$ by $S_0(\mu)=S^*_0$ and
\begin{equation}\label{eq:16}
    S_t(\mu) = S_{t-1}(\mu)+x^*_t(\mu), \qquad 1 \le t \le T.
\end{equation}
For each such $\mu$ define $\oT(\mu)$ be first time~$t$, $1\le t\le
T$, such that $S_t(\mu)$ violates one of the capacity
constraints~\eqref{eq:3}; if there is no such time (i.e.\ the path
$S(\mu)$ satisfies all the capacity constraints and so is feasible for
the problem~$\mathbf{P}$) we write $\oT(\mu)=\infty$.  Define $M_1$ to
be the set of $\mu$ such that $\oT(\mu)\le T$ and such that it is the lower
capacity constraint which is violated at the time $\oT(\mu)$ (i.e.\
$S_{\oT(\mu)}(\mu)<0$ if $\oT(\mu)<T$, and $S_{\oT(\mu)}(\mu)<S_T^*$ if
$\oT(\mu)=T$).  Similarly define $M'_1$ to be the set of $\mu$ such that
$\oT(\mu)\le T$ and such that it is the upper capacity constraint which is
violated at the time $\oT(\mu)$ (i.e.\ $S_{\oT(\mu)}(\mu)>E$ if
$\oT(\mu)<T$, and $S_{\oT(\mu)}(\mu)>S_T^*$ if $\oT(\mu)=T$).


Since each $x^*_t(\mu)$ is increasing in $\mu$, it follows that if
$\mu\in M_1$ then $\mu'\in M_1$ for all $\mu'<\mu$ and that if $\mu\in
M_1'$ then $\mu'\in M_1'$ for all $\mu'>\mu$; further the sets $M_1$
and $M_1'$ are disjoint, and (since the pair $(S^*,\mu^*)$ exists)
neither $M_1$ nor $M_1'$ can be the entire real line.  We now set
$\bar\mu_1=\sup M_1$.  (In the case where $M_1$ is empty---which could
only happen when the sole feasible strategy for the management of the
store would be to reduce its level by the maximum of $P_o$ at each
successive time~$t$, this being just sufficient to obtain the required
level $S^*_T$ at time~$T$---we could formally set
$\bar\mu_1=-\infty$).  Consider the behaviour of $S(\bar\mu_1)$, for
which there are three possibilities:
\begin{compactenum}[(a)]
\item the vector $S(\bar\mu_1)$ is feasible (i.e.\
  $\oT(\bar\mu_1)=\infty$); in this case we take $K=1$, the
  time~$T_1=T$, and $S^*_t=S_t(\bar\mu_1)$ with $\mu^*_t=\bar\mu_1$ for
  $1\le t\le T$;
\item the scalar $\bar\mu_1$ belongs to the set $M_1$; we here define
  $\oT_1=\oT(\bar\mu_1)$ and note that there necessarily exists at
  least one $t<\oT_1$ such that $S_t(\bar\mu_1)=E$ (for otherwise, by
  the continuity of each $S_t(\mu)$ in $\mu$, $\mu$ could be increased
  above $\bar\mu_1$ while still belonging to the set $M_1$); define
  $T_1$ to be any such $t$,
  and take $S^*_t=S_t(\bar\mu_1)$ and
  $\mu^*_t=\bar\mu_1$ for all $t$ such that $1\le t\le T_1$;
\item the scalar $\bar\mu_1$ belongs to the set $M_1'$; we here again
  define $\oT_1=\oT(\bar\mu_1)$ and note that, similarly to the case
  (b), there necessarily exists at least one $t<\oT(\bar\mu_1)$ such
  that $S_t(\bar\mu_1)=0$; define $T_1$ to be any such $t$,
  and again take $S^*_t=S_t(\bar\mu_1)$ and $\mu^*_t=\bar\mu_1$ for
  all $t$ such that $1\le t\le T_1$.
\end{compactenum}

The time $T_1$ and the constant $\bar\mu_1$ thus identified, the above
construction is now restarted at each of the successive times $T_k$,
$k=1,\dots,K-1$.  At each such time $T_k$ we replace $S^*_0$ by
$S^*_{T_k}$ and identify the corresponding sets $M_{k+1}$, $M'_{k+1}$,
the constant $\bar\mu_{k+1}$, and hence the times~$\oT_{k+1}$,
$T_{k+1}$.  We then set $\mu^*_t=\bar\mu_{k+1}$ and
$S^*_t=S^*_{t-1}+x^*_t(\bar\mu_{k+1})$ for $t=T_k+1,\dots,T_{k+1}$.
We continue thus until we obtain $k=K$ such that $T_K=T$.

In the more general case where the functions $C_t$ are not necessarily
strictly convex, we have the complication that, for appropriate $\mu$,
the quantity $x^*_t(\mu)$ may not be uniquely defined.  Rather each of
the ``functions'' $x^*_t$ can be viewed as a many-valued function
which is increasing in the sense that for $\mu_1<\mu_2$ we have
$x^*_t(\mu_1)\le x^*_t(\mu_2)$ for any values of $x^*_t(\mu_1)$ and
$x^*_t(\mu_2)$, and which is further continuous in the sense that (by
the supporting hyperplane theorem) every $x\in\X$ is a possible value
of $x^*_t(\mu)$ for some $\mu$.  In the first step of the above
construction (that required to identify the times $\oT_1$ and $T_1$
together with $(S^*_t,\mu^*_t)$ for $1\le t\le T_1$), these properties
of the many-valued functions $x^*_t$ extend in the obvious sense to
the paths $S(\cdot)$ given by \eqref{eq:16}, each of which now becomes
an envelope of paths.  Thus only obvious modifications are required in
order to proceed as before.  (The one formality is that the sets $M_1$
and $M'_1$ should be replaced by sets of paths, consisting of those
$S(\mu)$ which on first violating a capacity constraint do so
respectively below or above.)

We now have the following result.

\begin{theorem}
  \label{thm:const}
  Assume $\rho=1$.  Then the pair $(S^*,\mu^*)$ as given by the above
  recursive construction satisfies the conditions (i)--(iii) of
  Theorem~\ref{thm:ls}.  Further, the ``locality'' properties asserted
  at 1.\ and 2.\ above hold.
\end{theorem}

\begin{proof}
  Again suppose first that the functions $C_t$ are strictly convex.
  
  To show the first assertion of the theorem, note the conditions~(i)
  and (ii) of Theorem~\ref{thm:ls} are satisfied by construction and,
  for the condition~(iii) of Theorem~\ref{thm:ls}, it only remains to
  show that, in the case $K\ge2$, the condition~\eqref{eq:5} of (iii)
  is satisfied for $t=T_1,\dots,T_{K-1}$.  It is sufficient to
  consider $t=T_1$.  Since we are assuming $K\ge2$, the first of the
  three possible behaviours for the vector~$S(\bar\mu_1)$ considered
  at (a)--(c) above cannot occur.  Thus, without loss of generality,
  assume $\bar\mu_1\in M_1$.  Then $0\le S_t(\bar\mu_1)\le E$ for
  $1\le t\le \oT_1-1$, while $S_{\oT_1}(\bar\mu_1)$ violates the
  capacity constraints below (i.e.\ $S_{\oT_1}(\bar\mu_1)<0$ if
  $\oT_1<T$ and $S_{\oT_1}(\bar\mu_1)<S^*_T$ if $\oT_1=T$); further,
  as already noted in the above construction, at the time $T_1<\oT_1$
  we have $S_{T_1}(\bar\mu_1)=S^*_{T_1}$.  Thus, considering the
  construction restarted at the time~$T_1$, it now follows that also
  $\bar\mu_1\in M_2$.  Hence, from the definition of $\bar\mu_2$, it
  follows that $\bar\mu_2\ge\bar\mu_1$ as required.

  For the second part of the theorem, we again assume $K\ge2$
  (otherwise there is nothing to show).  Once more, it is sufficient
  to consider $k=1$.  Observe that, in the above construction,
  $\oT_1(\mu)$ is increasing in $\mu$ for $\mu\in M_1$ and decreasing in
  $\mu$ for $\mu\in M_1'$.  Suppose, without loss of generality,
  $\bar\mu_1\in M_1$.  Then, again from the above construction,
  $\oT_1(\mu)\le\oT_1$ for all $\mu\in M_1$ and $\oT_1(\mu)\le
  T_1<\oT_1$ for all $\mu\in M_1'$, so that the asserted result follows.
  
  In the case where the functions $C_t$ are not necessarily strictly
  convex, again only obvious and formal modifications are required: we
  proceed as indicated earlier, replacing the space of possible $\mu$
  with the space of possible paths $S(\mu)$ (where there may be
  infinitely many $S(\mu)$ corresponding to particular values of
  $\mu$). 
\end{proof}

\paragraph{Algorithm.}  Theorem~\ref{thm:const} gives an algorithm for
the construction of the pair $(S^*,\mu^*)$.  This algorithm is local
in time in the sense which is made precise in the statement of that
theorem, but which may be stated informally as being such that the
determination of the optimal control at any time depends only on a
knowledge of future cost functions to a time horizon which may be well
short of the final time $T$.  As previously remarked it is thus
typically suitable for the management of a store on an infinite time
horizon.  However, in the numerical implementation of the algorithm
there are some considerations which are worth commenting on at this
point.  We again focus on the first step of the algorithm in which,
given the initial level $S^*_0$ of the store, it is required to
determine the time $T_1$ and the value $\bar\mu_1$ (such that
$\mu^*_t=\bar\mu_1$ and $S^*_t=S(\bar\mu_1)$ for $1\le t\le T_1$).

In the case where the cost functions $C_t$ are strictly convex, the
determination of $\bar\mu_1$ usually---and inevitably in the case of
general convex cost functions---involves some form of numerical search
(e.g.\ a simple binary search) which terminates with a pair of values
$\bar\mu_1^l\in M_1$ and $\bar\mu_1^u\in M'_1$ such that
$\bar\mu_1^l<\bar\mu_1^u<\bar\mu_1^l+\epsilon$ to within some
sufficiently small tolerance $\epsilon>0$.  Suppose, without loss of
generality, that $\oT_1(\bar\mu_1^l)>\oT_1(\bar\mu_1^u)$.  It then
follows from the continuity in $\mu$ of the sample paths $S(\mu)$ that
at the time $t=\oT_1(\bar\mu_1^u)$ we have $S_t(\bar\mu_1^l)\approx
S_t(\bar\mu_1^u)\approx E$ (the errors in the approximations being
$o(\epsilon)$ as $\epsilon\to0$).  Thus, revisiting the detail of the
proof of Theorem~\ref{thm:const}, it is easy to see that we may make
the approximation $\bar\mu_1=\bar\mu_1^u$ (or $\bar\mu_1=\bar\mu_1^l$)
and $T_1=\oT_1(\bar\mu_1^u)$.  Similarly in the case where
$\oT_1(\bar\mu_1^l)<\oT_1(\bar\mu_1^u)$ we may take
$T_1=\oT_1(\bar\mu_1^l)$.  The error in the ultimately constructed
pair $(S^*,\mu^*)$ is then again $o(\epsilon)$ as $\epsilon\to0$.

In the case where the cost functions $C_t$ are not necessary strictly
convex, more care is as usual required, and a numerical search
terminates when we obtain a pair of paths of the form $S(\mu)$---one
first violating a constraint below and the other first violating a
constraint above---which are sufficiently close to each other.  It is
here possible that these paths may correspond to the same value of
$\mu$.  Thus those values of $\mu$ such that, for some $t$,
$x^*_t(\mu)$ is nonunique typically require to be identified in
advance.  Finally we remark that in the case where the cost functions
$C_t$ are simply piecewise linear (as in the ``small store'', or
price-taker, case in which the cost functions~$C_t$ are given by
\eqref{eq:1}), then the above algorithm may be adapted to avoid
numerical search.  Alternatively, standard linear programming
techniques may of course be used in this case, though it is not
obvious how these might be adapted to yield the ``time locality''
property which is identified above and which permits the optimal
control of the store on essentially infinite time horizons.

\paragraph{The case $\rho\le1$.}

We now consider briefly the case of general $\rho\le1$, i.e.\ where
we also model possible leakage from the store.  Only small and readily
understood modifications are required to the above algorithm.  Here,
as before, the essence of the argument is to attempt to choose
$(S^*,\mu^*)$ so as to satisfy the conditions of Theorem~\ref{thm:ls},
again by choosing the components of these vectors successively in
time, but now maintaining the relationship $\rho\mu^*_{t+1} =
\mu^*_t$, except at those times~$t$ such that the store is either
empty or full.  Thus we proceed as previously, except that the
relation~\eqref{eq:16} now becomes
\begin{displaymath}
    S_t(\mu) = \rho S_{t-1}(\mu)+x^*_t(\rho^{1-t}\mu), \qquad 1 \le t \le T,
\end{displaymath}
and corresponding and obvious small modifications are required in the
three cases (a)--(c) considered previously.

\paragraph{Further discussion.}

In the above construction, the typical length of the intervals between
the successive times $T_i$ depends on the shape of the cost functions
$C_t$ (notably the difference between buying and selling prices),
together with the rate at which these functions fluctuate in time.
This is to be expected as the store operates by selling at prices
above those at which it bought, and what is important is the frequency
with which such events can occur.  For example, such fluctuations may
occur an a 24-hour cycle, and, depending on the shape of the cost
functions, the typical length of the intervals between the successive
times $T_i$ may then be of the order of around 12 hours.  These points
are illustrated further in the examples of Section~\ref{sec:example}.

Finally we remark that, again in the above construction, it is not
difficult to see that, for each $k\le K-1$, suitable variation of the
cost function~$C_{\oT_k}$ changes $(S^*_t,\mu^*_t)$ for $T_{k-1}+1\le t\le
T_k$, and further that $\oT_1\le\dots\le\oT_K$.  Thus the latter
sequence provides, in the obvious sense, a running minimal time
horizon for the algorithmic solution of the problem~$\mathbf{P}$, and
in this sense the above algorithm is optimal.



\section{Sensitivity of store value with respect to constraint variation}
\label{sec:sens-store-value}

Under suitable differentiability assumptions, the Lagrangian theory of
the preceding sections enables an immediate determination of the effect
on the cost of operating the store (the negative of its value) of
marginal variations in either the capacity or the rate constraints.
The capacity variation result is almost immediate, while the rate
constraint result requires a modest extension of the earlier theory.
Throughout we again consider the more general
problem~$\mathbf{P}(a,\,b)$ introduced in Section~\ref{sec:solution},
together with its minimised objective function
$V(a,\,b)$---corresponding to the minimum cost of operating the store.
We again let $a^*$ and $b^*$ to be the values of $a$ and $b$
corresponding to our particular problem~$\mathbf{P}$ of interest---as
previously defined.  We assume throughout this section that the
minimised objective function $V(a,\,b)$ is differentiable with respect
to (each of the components of) the vectors~$a$ and $b$ at
$(a^*,\,b^*)$---as will be the case when, for example, the cost
functions~$C_t$ are differentiable at the solution to the
problem~$\mathbf{P}$.

Under this differentiability condition the vector~$\mu^*$ of
Theorem~\ref{thm:ls} is uniquely defined.  This follows from
consideration of the algorithm of Section~\ref{sec:algorithm}, which
sequentially constructs a pair $(S^*,\,\mu^*)$ satisfying the
conditions of Theorem~\ref{thm:ls}.  Here the differentiability
condition above implies easily that any attempt to vary $\mu^*$ as
constructed by that algorithm leads to a violation of the
complementary slackness conditions~(iii) of Theorem~\ref{thm:ls}.
(Alternatively, the uniqueness may here be argued directly from the
conditions~(ii) and (iii) of Theorem~\ref{thm:ls}, again by
considering infinitesimal variation of $\mu^*_t$ at those times~$t$
such that the capacity constraints are binding.)  This vector~$\mu^*$
is thus as identified by Theorem~\ref{thm:exist}---and has the
interpretation in terms of Lagrange multipliers given there---and is
as constructed by the algorithm of Section~\ref{sec:algorithm}.


It is convenient to write $V^*$ for the value $V(a^*,\,b^*)$ of the
minimised objective function for our particular problem of interest
$\mathbf{P}=\mathbf{P}(a^*,\,b^*)$.  For the sensitivity of the cost
of operating the store with respect to variation in the capacity
constraint, we have the following result.





\begin{theorem}
  \label{thm:capacity-variation}
  The derivative of the cost of operating the store with respect to
  variation of the capacity $E$ is given by
  \begin{equation}
    \label{eq:17}
    \frac{\partial V^*}{\partial E}
    = \sum_{t\in\tau} (\mu^*_t-\rho\mu^*_{t+1}),
  \end{equation}
  where $\tau$ is the set of times~$t$ such that $1\le t\le T-1$ and
  $S^*_t=E$, and where $\mu^*$ is as identified above.
\end{theorem}

\begin{proof}
  Let $\alpha^*$ and $\beta^*$ be the vector Lagrange multipliers
  introduced in the proof of Theorem~\ref{thm:exist}.  Recall also the
  definition of $b^*$ above.  From the standard interpretation of
  Lagrange multipliers in the presence of differentiability of an
  objective function,
  \begin{align}
    \frac{\partial V^*}{\partial E}
    & = \sum_{1\le t\le T-1} \beta^*_t \nonumber\\
    & = \sum_{t\in\tau} (\alpha^*_t + \beta^*_t),
    \label{eq:18}
  \end{align}
  where \eqref{eq:18} above follows from the conditions~\eqref{eq:11}
  and \eqref{eq:12} (which imply that for $1\le t\le T-1$, we have
  $\beta^*_t=0$ for $t\notin\tau$ and $\alpha^*_t=0$ for $t\in\tau$).
  The required result now follows on using \eqref{eq:15}.
\end{proof}


We now consider the sensitivity of the cost of operating the store
with respect to variation in the rate constraints.  We here have the
following result.

\begin{theorem}
  \label{thm:rate-variation}
  Assume additionally that the cost functions $C_t$ are differentiable
  at the points $P_i$ and $-P_o$ corresponding to the input and output
  rate constraints.  Then the derivatives of the cost $V(a^*,\,b^*)$
  of operating the store with respect to variation of the input and
  output rate constraints $P_i$ and $P_o$ are given respectively by
  \begin{align}
    \frac{\partial V^*}{\partial P_i}
    & = \sum_{t\in\tau_i} (C'_t(P_i) - \mu^*_t)\label{eq:19}\\
    \frac{\partial V^*}{\partial P_o}
    & = \sum_{t\in\tau_o} (\mu^*_t - C'_t(-P_o)), \label{eq:20}
  \end{align}
  where $\tau_i$ is the set of times~$1\le t\le T$ such that
  $x_t(S^*)=P_i$ and $\tau_o$ is the set of times~$1\le t\le T$ such
  that $x_t(S^*)=-P_o$ (i.e.\ $\tau_i$ and $\tau_o$ are respectively
  the sets of times such that the input and output rate constraints
  are binding at the solution $S^*$ to the problem~$P$), and where
  again $\mu^*$ is as identified above.
\end{theorem}

\begin{proof}
  We proceed as in the proof of Theorem~\ref{thm:exist}.  However, we
  rewrite the problem~$\mathbf{P}(a,\,b)$ by relaxing the rate
  constraints $x_t(S)\in\X$ to $x_t(S)\in\R$ and introducing instead
  the additional functional constraints
  \begin{align}
    x_t(S) + u_t & = P_i, \qquad 1 \le t \le T, \label{eq:21}\\
    x_t(S) - v_t & = -P_o, \qquad 1 \le t \le T, \label{eq:22}
  \end{align}
  for slack (or surplus) variables $u=(u_1,\dots,u_T)$ and
  $v=(v_1,\dots,v_T)$ constrained to be positive.  We thus introduce
  additional vectors $\gamma^*=(\gamma^*_1,\dots,\gamma^*_T)$ and
  $\delta^*=(\delta^*_1,\dots,\delta^*_T)$ of Lagrange multipliers to
  deal respectively with the additional functional
  constraints~\eqref{eq:21} and \eqref{eq:22}.  Arguing as before we
  have the further complementary slackness conditions (in addition to
  \eqref{eq:11} and \eqref{eq:12})
    \begin{align}
    \gamma^*_t \le 0, \qquad &
    \text{$\gamma^*_t=0$  whenever $u^*_t>0$},
    \qquad 1\le t \le T,
    \label{eq:23}\\
    \delta^*_t \ge 0, \qquad &
    \text{$\delta^*_t=0$  whenever $v^*_t>0$},
    \qquad 1\le t \le T.
    \label{eq:24}
  \end{align}
  where $u^*$ and $v^*$ are the values of $u$ and $v$ at the solution
  $S^*$ to the original problem~$\mathbf{P}$.  Again arguing as in the
  proof of Theorem~\ref{thm:exist}, we now have that, for each $1\le
  t\le T$,
  \begin{equation}
    \text{$x_t(S^*)$ minimises
      $C_t(x)-(\mu^*_t+\gamma^*_t+\delta^*_t)x$ in $x\in\R$,}
      \label{eq:25}
  \end{equation}
  where the vector~$\mu^*=(\mu^*_1,\dots,\mu^*_T)$ remains as
  identified in Theorem~\ref{thm:exist}---since the interpretations as
  derivatives of the Lagrange multipliers $\alpha^*$ and $\beta^*$ of
  that theorem remain unchanged and $\mu^*$ remains as identified by
  \eqref{eq:15}.  (We observe in passing that the
  relation~\eqref{eq:25} stands formally in contrast to the result in
  the proof of Theorem~\ref{thm:exist} where, from \eqref{eq:14},
  $x_t(S^*)$ minimised $C_t(x)-\mu^*_tx$ in $x\in\X$).

  We now note that, once again from the differentiability assumptions
  of the present theorem, and standard Lagrangian theory,
  \begin{displaymath}
    \frac{\partial V^*}{\partial P_i} = \sum_{t\in\tau_i} \gamma^*_t.
  \end{displaymath}
  Further, for $t\in\tau_i$, we have $v^*_t=P_i+P_o>0$ and so
  $\delta^*_t=0$ (from \eqref{eq:24}) and also
  $C'_t(P_i)=\mu^*_t+\gamma^*_t$ (from \eqref{eq:25}).  The
  result~\eqref{eq:19} now follows.  The result~\eqref{eq:20} follows
  similarly.
 \end{proof}

 \begin{remark}
   Note that the results~\eqref{eq:19} and \eqref{eq:20} of
   Theorem~\ref{thm:rate-variation} are also intuitively clear from
   the interpretation of $\mu^*_t$ given in Section~\ref{sec:solution}
   as a notional unit reference value for additions to the store at
   each time~$t$.  Thus for \eqref{eq:19}, note that, for each
   $t\in\tau_i$, increasing the maximum input rate~$P_i$ by
   $\mathrm{d}P_i$ permits the addition of increased value
   $\mu^*_t\mathrm{d}P_i$---corresponding to the addition to the level
   of the store---at a cost of $C'_t(P_i)\mathrm{d}P_i$.
 \end{remark}

\section{Examples}
\label{sec:example}
In this section we illustrate some of our results with an example
storage facility which has market impact.
We use half-hourly time units and a cost series $(p_1,\ldots,p_T)$
corresponding to the real half-hourly spot market wholesale
electricity prices in Great Britain for the year 2011.  As might be
expected these prices show a strong daily cyclical behaviour. We
assume that the store is large enough to have market impact on prices,
but small enough in relation to the rest of the network that the price
at which the store buys or sells energy can be approximated by a
linear function of the amount of energy traded by the store.  The
resulting cost function is quadratic and of the form
\begin{equation}
  \label{eq:quadcost}
  C_t(x) =
  \begin{cases}
    (p_t+p_t' x)x & \quad\text{if $x\ge0$}\\
    (p_t+\eta p_t' x)\eta x  & \quad\text{if $x<0$}
  \end{cases}
\end{equation}
where $\eta$ is the time-independent, or round-trip, efficiency of the
store and $p_t'\geq 0$ is a measure of the market impact of the store
on the price at time~$t$.  The terms in brackets
in~(\ref{eq:quadcost}) are the prices which result from filling (or
emptying) the store by $x$ units of energy.  In the following
examples, we assume further that each $p_t'$ is proportional to the
wholesale price $p_t$ at that time, so that $p_t'=\lambda p_t$ for
some $\lambda\geq 0$.  This reflects the intuition that the market
becomes more price-responsive when prices are high.  The special case
$\lambda=0$ corresponds to the price-taking store with cost
function~(\ref{eq:1}).  We assume a common input and output rate
constraint $P_i=P_o=P$ and, as before, denote by $E$ the capacity of
the store.  Finally, while we allow a round-trip efficiency $\eta<1$,
we assume throughout that there is no leakage from the store over
time, i.e.\ that $\rho=1$.

\begin{figure}[!h]
  \centering
  \includegraphics[scale=0.9]{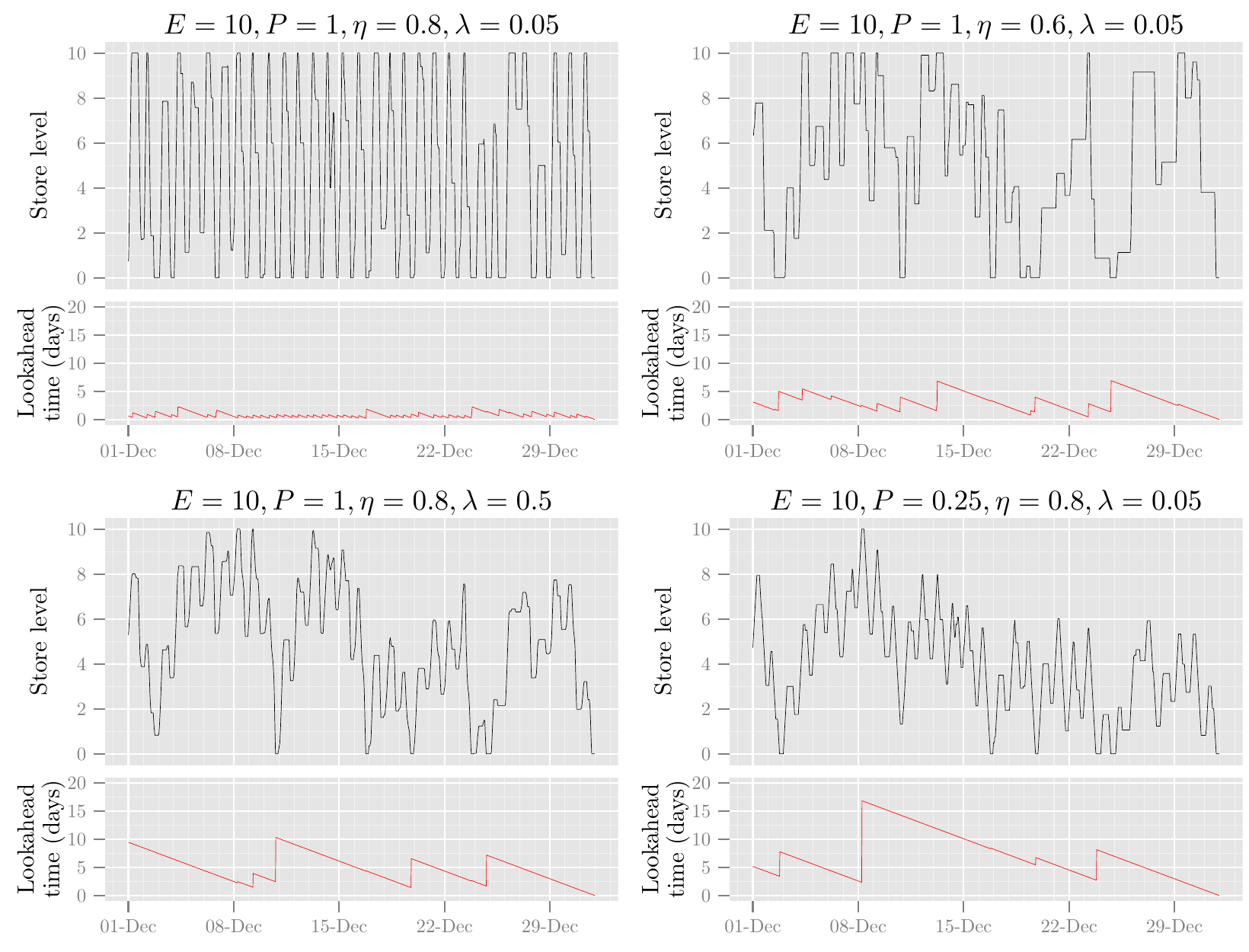}
  \caption{Examples in which the parameters associated with the store
    are varied.  In each case, the upper plot shows the optimal level
    of storage and the lower plot shows the look-ahead time required
    at each stage of the optimisation.}
  \label{fig:A}
\end{figure}

The optimal strategy associated with the cost
function~(\ref{eq:quadcost}) is shown in Figure~\ref{fig:A} (the upper
plot in each quadrant) for various choices of parameters.  The
optimisation takes place over the whole year and we present here the
behaviour of the store over a single month (December).  The plot in
the top-left quadrant corresponds to a ``base'' case, with the
parameter choices $E=10$, $P=1$, $\eta=0.8$ and $\lambda=0.05$.  The
time $E/P=10$ half-hours units for the store to completely fill or
empty and the round-trip efficiency of $0.8$ correspond approximately
to the Dinorwig pumped storage facility in Snowdonia in North Wales;
since, in the units of this example, the maximum volume which can be
bought or sold in a single period is $1$, the choice $\lambda=0.05$
indicates only modest market impact.  The upper portion of the plot
shows the variation of the store level with time~$t$, while the lower
portion shows, for each time~$t$, the time horizon $\oT_k-t$, where
$k$ is such that $T_{k-1}+1\le t\le T_k$, defined in
Section~\ref{sec:algorithm}; the latter is the length of time into the
future over which it is necessary to examine the cost functions in
order to make the optimal decision at time~$t$.  It is seen that,
under the optimal strategy, the store usually completely empties and
fills on a daily cycle, with some lull in activity over the Christmas
period.  As might be expected the time horizon necessary for an
optimal decision is of the order of a day or so.

The plots in the remaining three quadrants of Figure~\ref{fig:A} are
each formed by varying one of the parameters of the base case example,
in each case in such a way that the store is less active.  The plot in
the top-right quadrant corresponds to a reduction in the round-trip
efficiency of the store from $\eta=0.8$ to $\eta=0.6$.  Here it is
seen that the store level cycles less frequently and tends to remain
at the same value for longer periods of time than in the base
case---as might be expected; the time horizons necessary for optimal
decision making are significantly longer than in the base case.  The
plot in the lower-left quadrant corresponds to an increase in the
``market impact'' factor from $\lambda=0.05$ to $\lambda=0.5$, while
that in the lower-right quadrant corresponds to a tightening of the
rate constraint from $P=1$ to $P=0.25$.  In both cases the store is
almost continuously active but trades at lower volumes than in the
base case; consequently time horizons for optimal decision making are
very much longer than in the base case.  The broad similarity of the
behaviour in these two examples may be explained by noting that an
increased market impact factor acts to slow down the activity rate of
the store in much the same way as a tightening of the rate constraint.
This is because buying prices increase in proportion to the market
impact factor with each additional unit of energy bought at that time,
whilst selling prices similarly decrease with energy sold.  The store
therefore needs to balance the benefit of operating at high powers
with the impact this has on prices.


For some further numerical results in the context of this particular
example, see~\cite{CGZ}.

\section{Stochastic models}
\label{sec:stochastic-model}

In practice there is uncertainty as to future energy prices, and hence
there is a need to consider models in which the cost functions~$C_t$
evolve randomly in time.  However, the temporal behaviour of such
prices may be very heterogeneous and unlikely to evolve in any
stochastically regular manner; thus any comprehensive stochastic
modelling of possible future behaviour, together with its optimisation
(which under such general circumstances would typically and
necessarily involve some form of stochastic dynamic programming) is
likely in practice to prove at least computationally infeasible.  Thus
we should wish to make some form of approximation, sufficiently good
as to work well at any time in determining the decision over the next
time step; after each such step the future could then be reassessed
and the control re-optimised.  

There is substantial evidence in the literature that this approach,
sometimes referred to as the ``rolling intrinsic policy'', often
works very well in practice, providing near-optimal strategies at a
much lower computational cost than dynamic programming and other
competing methods (see, for example, \cite{LMS} for a comparison of
different approximate optimisation methods, both in terms of
computational efficiency and accuracy).  Examples of cost
distributions which have been handled using this approach in the
literature, and shown to produce near-optimal results, include (gas)
prices whose logarithms evolve as a single-factor, mean-reverting
stochastic process~\cite{Sec2010}, and prices which are characterised
by multivariate driftless Brownian motions~\cite{LMS,WWQ}.
In~\cite{SDJW}, a back-casting approach is employed, which can be
considered as a special case of the rolling intrinsic policy, in which
at each stage of re-optimisation, past prices (from the previous two
weeks) are used as future prices.  Even under this relatively simple
regime, it is illustrated that a store could gain between 80 and
90$\%$ of the profit available in a deterministic setting.


In the present section we propose a stochastic model, in which future
uncertainty has a martingale structure (which seems a plausible first
approximation to a stochastic structure for price uncertainty).  We
show that for this model the exact optimal policy is simply that for
the deterministic model in which future cost functions are replaced by
their expected values, and may thus be determined as in
Section~\ref{sec:algorithm}.  In a more general stochastic setting, we
propose the following relatively simple strategy: successively at each
time step, future cost functions are replaced by their expected values
and the present algorithm then used to work out how much to buy or
sell in the next time step; future expected cost functions are then
re-evaluated prior to the next step.  We expect this method to work
well, provided that the future expected cost functions, as seen at
each re-optimization time $t$, are sufficiently close to the actual
costs up until the first time horizon $\oT_k$ which follows $t$ (where
$\oT_k$ is as defined previously).
In particular, our analysis in Section~\ref{sec:algorithm} shows that,
if expected costs exactly match actual costs between times $t$ and
$\oT_k$, then any uncertainty in costs after $\oT_k$ are
irrelevant to the decision of the store at time $t$---thus, any
inaccuracies arising from this approach are due only to forecasting
inaccuracies between times $t$ and $\oT_k.$.  Given also the relative
computational efficiency of the current algorithm, in particular its
identification of the shortest time horizon required for the
determination of the optimal decision at each time step, we believe
that this method should provide a near-optimal procedure for the
efficient real-time management of storage over extended periods of
time.

Thus we consider a model in which uncertainties in future costs evolve
multiplicatively as we proceed backwards in time.  (This seems a
possible first approximation to market uncertainty.)
More precisely we assume that the cost functions $C_t$ are given by
\begin{displaymath}
  C_t = \xi_t \bar C_t, \qquad 1 \le t \le T,
\end{displaymath}
where $(\bar C_1,\dots,\bar C_T)$ is a sequence of deterministic cost
functions and where $(\xi_1,\dots,\xi_T)$ is a sequence of strictly
positive real-valued random variables forming a martingale, i.e.\ such
that
\begin{equation}
  \label{eq:26}
  \E(\xi_t\,|\,\mathcal{F}_{t-1}) = \xi_{t-1}, \qquad 1 \le t \le T;
\end{equation}
here $\E$ denotes expectation and each $\mathcal{F}_t$ is the
$\sigma$-algebra generated by $\xi_1,\dots,\xi_t$ (with
$\mathcal{F}_0$ the trivial $\sigma$-algebra).  Note that, since the
functions $\bar C_t$ may if necessary be rescaled, there is no loss of
generality in omitting a multiplicative constant from \eqref{eq:26}.
The deterministic functions $\bar C_t$ are assumed to satisfy the same
conditions as the cost functions $C_t$ of the deterministic problem
given in Section~\ref{sec:problem}, and hence the random cost
functions $C_t$ also satisfy these conditions.

The optimization problem~$\mathbf{P}$ of Section~\ref{sec:problem} now
becomes 
\begin{compactitem}
\item[$\mathbf{P}$:] choose the random vector $S=(S_1,\dots,S_T)$, with
  $S_t\in\mathcal F_t$ for each $t$, so as to minimise
  \begin{equation}\label{eq:27}
    G(S) := \E\left[\sum_{t=1}^T C_t(x_t(S))\right]
  \end{equation}
  with $S_0=S^*_0$ and $S_T=S^*_T$ (where $S^*_0$ and $S^*_T$ are
  fixed constants as previously), and again subject to the capacity
  constraints
  \begin{displaymath}
    0 \le S_t\le E,
    \qquad 1 \le t \le T-1.
  \end{displaymath}
  and the rate constraints
  \begin{displaymath}
    x_t(S) \in \X,
    \qquad 1 \le t \le T.
  \end{displaymath}
\end{compactitem}
Note in particular that each $S_t$ (or, equivalently, each $x_t(S)$)
may be chosen based on the knowledge of the realised random variables
$\xi_1,\dots,\xi_t$ up to time $t$.  We now have the following result
(which we reiterate one would expect to use in practice by coupling it
with re-optimisation at each time step).

\begin{theorem}
  \label{thm:stochastic}
  The solution to the above problem remains deterministic, with the
  optimal sequence of store levels as given in the case where
  stochastic cost functions $C_t$ are replaced by their deterministic
  counterparts $\bar C_t$.  Further the optimized value of the
  objective function~\eqref{eq:27} is the same as that for the
  deterministic variant of the problem.
\end{theorem}

\begin{remark}
  This result is intuitively clear, since the stochastic aspect of the
  problem can be characterised as consisting of, at each successive
  time, a random but uniform scaling of all future costs, and any
  such scaling cannot change the optimal strategy.  However, a formal
  proof is required.
\end{remark}

\begin{proof}[Proof of Theorem~\ref{thm:stochastic}]
  Consider first the case in which the stochastic cost functions $C_t$
  are replaced by their deterministic counterparts $\bar C_t$.  For
  each $0\le t\le T-1$, and each fixed $S_t$ such that $0\le S_t\le
  E$, with $S_0=S^*_0$, define
  \begin{displaymath}
    \bar V_t(S_t) = \min_{S_{t+1},\dots,S_{T-1}}\sum_{u=t+1}^T\bar C_u(x_u(S)),
  \end{displaymath}
  where $S=(S_t,\dots,S_T)$ and, for each $u>t$, we have $0\le S_u\le E$
  with $S_T=S^*_T$ and where $x_u(S)=S_u-\rho S_{u-1}$ satisfies the
  rate constraint~$x_u(S)\in\X$.  Define also
  $\bar V_T(S^*_T)=0$.  Thus $\bar V_t(S_t)$ represents optimised
  future costs at time $t$ given that the level of the store is then
  $S_t$.  Then, by the usual dynamic programming recursion, we have
  \begin{equation}
    \label{eq:28}
    \bar V_t(S_t)
    = \min_{x_{t+1}\in\X}\left[\bar C_{t+1}(x_{t+1})
      + \bar V_{t+1}(\rho S_t + x_{t+1})\right],
    \qquad 0 \le t \le T-1,
  \end{equation}
  where the above minimisation is taken over $x_{t+1}\in\X$ such that
  $0\le\rho S_t+x_{t+1}\le E$ for $0\le t\le T-2$ and $\rho
  S_{T-1}+x_T=E$.

  In the general stochastic case define similarly, for $0\le t\le
  T-1$, and each fixed $S_t$ such that $0\le S_t\le E$, again with
  $S_0=S^*_0$,
  \begin{equation}
    \label{eq:29}
    V_t(S_t)
    = \E\left[\min_{S_{t+1},\dots,S_{T-1}}\sum_{u=t+1}^T
      C_u(x_u(S))\,\Bigg\lvert\,\F_t\right],
  \end{equation}
  where the random vector $S=(S_t,\dots,S_T)$ and, for each $u>t$, we
  have $S_u\in\F_u$ and $0\le S_u\le E$ with $S_T=S^*_T$ and where
  $x_u(S)=S_u-\rho S_{u-1}\in\X$.  Define also $V_T(S^*_T)=0$.  Thus
  again $V_t(S_t)$ represents optimised future costs at time $t$ given
  that the level of the store is then $S_t$.

  We now assert that, for each $t$ and $S_t$ as above,
  \begin{equation}
    \label{eq:30}
    V_t(S_t)=\xi_t\bar V_t(S_t).
  \end{equation}
  The proof of this assertion is by backwards induction in time $t$.
  The result is trivially true for $t=T$.  Assume now that it is true
  for $t=u+1$, where $0\le u\le T-1$.  Then, analogously
  to~\eqref{eq:28},
  \begin{align}
    V_u(S_u)
    & = \E\left[\min_{x_{u+1}\in\F_{u+1}} [C_{u+1}(x_{u+1})
      + V_{u+1}(\rho S_u + x_{u+1})]\,\bigg\lvert\,\F_u\right]
    \nonumber\\
    & = \E\left[\min_{x_{u+1}\in\F_{u+1}} \xi_{u+1}[\bar C_{u+1}(x_{u+1})
      + \bar V_{u+1}(\rho S_u + x_{u+1})]\,\bigg\lvert\,\F_u\right]
    \nonumber\\
    & = \E\left[\xi_{u+1}\bar V_u(S_u)\,\bigg\lvert\,\F_u\right]
    \label{eq:31}\\
    & = \xi_u\bar V_u(S_u), \label{eq:32}
  \end{align}
  where the above minimisation is taken over $x_{u+1}\in\F_{u+1}$,
  $x_{u+1}\in\X$, and such that $0\le\rho S_u+x_{u+1}\le E$ with $\rho
  S_{T-1}+x_T=E$ in the case $u=T-1$, and where \eqref{eq:31} and
  \eqref{eq:32} follow from \eqref{eq:26} and \eqref{eq:30}
  respectively.  Hence the assertion~\eqref{eq:30} holds for all $t$
  and for all $S_t$.

  Note also that, from iteration of the argument leading to
  \eqref{eq:32}, for each $t$ and $S_t$, the optimising values of
  $S_{t+1},\dots,S_{T-1}$ are as in the deterministic case.  The
  theorem now follows from this observation and from~\eqref{eq:30} in
  the case $t=0$.
\end{proof}

\section{Commentary and conclusions}
\label{sec:conclusions}

In the preceding sections we have developed the optimization theory
associated with the use of storage for arbitrage, in particular the
strong Lagrangian theory which may be used to form the basis of
optimal control and which is necessary for the correct dimensioning of
storage facilities.  We have also given an algorithm for the
determination of the optimal control policy and of the associated
Lagrange multipliers.  In particular the algorithm captures the fact
that the control policy is essentially local in time, in that, for a
given system subject to given capacity and rate constraints, at each
time optimal decisions are dependent only on future cost functions
within an identifiable and typically short time horizon.

Our framework accounts for nonlinear cost functions, rate constraints,
storage inefficiencies, and the effect of externalities caused by the
activities of the store impacting the market.  It further accounts for
leakage over time from the store---something which may be expected to
substantially further localise over time the character of optimal
control policies.  While the model of the earlier sections of the
paper is deterministic in that it assumes that all the prices
determining the cost functions are known in advance, we have also
considered what we hope to be a realistic approach to near-optimal
control in a stochastic cost environment: the formulation of a
reasonably realistic approximate model for which the optimal control
may be precisely and efficiently evaluated via the earlier
deterministic algorithm, combined with the ability to re-optimise at
each time step by reformulating the approximation.  This general
approach has been shown to work well elsewhere.

What we have not done in the present paper is to consider the use of
storage for providing a reserve in case of unexpected system shocks,
such as sudden surges in demand or shortfalls in supply.  This problem
is considered by other authors (see, for example, \cite{BGK,GTL,GLTP})
in the case where the probabilities of storage underflows or overflows
are controlled to fixed levels.  However, we believe that a further
approach here would be to attach economic values to such underflows or
overflows, translating to attaching an economic worth to the absolute
level the store (as opposed to attaching a worth to a \emph{change} in
the level of the store as in the present paper).  Since in practice
storage is used both for arbitrage and for buffering or control as
described above, this would provide a more integrated approach to the
full economic valuation of such storage.

\section*{Acknowledgements}
\label{sec:acknowledgements}
The authors wish to thank their co-workers Andrei Bejan, Janusz
Bialek, Chris Dent and Frank Kelly for very helpful discussions during
the preliminary part of this work.  They are also most grateful to the
Isaac Newton Institute for Mathematical Sciences in Cambridge for
their funding and hosting of a number of most useful workshops to
discuss this and other mathematical problems arising in particular in
the consideration of the management of complex energy systems.  Thanks
also go to members of the IMAGES research group, in particular Michael
Waterson, Robert MacKay, Monica Giulietti and Jihong Wang, for their
support and useful
discussions.  The authors are further grateful to
National Grid plc for additional discussion and the provision of data,
and finally to the Engineering and Physical Sciences Research Council
for the support of the research programme under which the present
research is carried out.

\newpage  

\end{document}